\documentclass[12pt,reqno]{amsart}

\usepackage{amssymb,amsthm}
\usepackage{amsfonts,cmtiup,comment,stmaryrd}

\usepackage{mathrsfs,cmtiup}

\makeatletter
\def\blfootnote{\xdef\@thefnmark{}\@footnotetext}
\makeatother

\newtheorem{theorem}{Theorem}[section]
\newtheorem{lemma}[theorem]{Lemma}
\newtheorem{prop}[theorem]{Proposition}
\newtheorem{corollary}[theorem]{Corollary}

\theoremstyle{definition}

\newtheorem{remark}[theorem]{Remark}
\newtheorem*{definition*}{Definition}

\newcommand{\N}{\mathbb N}
\newcommand{\F}{\mathbb F}

\newcommand{\f}{\varphi}

\newcommand{\g}{\gamma }

\renewcommand{\geq}{\geqslant}
\renewcommand{\leq}{\leqslant}

\newcommand{\ed} {\end{document}}

\let\leq=\leqslant
\let\geq=\geqslant
\numberwithin{equation}{section}

\begin{document}
\title{On finite groups with an automorphism of prime order\\ whose fixed points have bounded Engel sinks}

\author{E. I. Khukhro}
\address{Charlotte Scott Research Centre for Algebra, University of Lincoln, U.K., and \newline \indent  Sobolev Institute of Mathematics, Novosibirsk, 630090, Russia}
\email{khukhro@yahoo.co.uk}

\author{P. Shumyatsky}

\address{Department of Mathematics, University of Brasilia, DF~70910-900, Brazil}
\email{pavel@unb.br}

\keywords{Finite groups; Engel condition; Fitting subgroup; automorphism}
\subjclass[2010]{20D45, 20D25, 20F45}

\begin{abstract}
A left Engel sink of an element $g$ of a group $G$ is a set ${\mathscr E}(g)$ such that for every $x\in G$ all sufficiently long commutators $[...[[x,g],g],\dots ,g]$ belong to ${\mathscr E}(g)$.  (Thus, $g$ is a left Engel element precisely when we can choose ${\mathscr E}(g)=\{ 1\}$.) We prove that if a finite group $G$ admits an automorphism $\f $ of prime order coprime to $|G|$ such that for some positive integer $m$ every element of the centralizer $C_G(\f )$ has a left Engel sink of cardinality at most $m$, then the index of the second Fitting subgroup $F_2(G)$ is bounded in terms of $m$.

A right Engel sink of an element $g$ of a group $G$ is a set ${\mathscr R}(g)$ such that for every $x\in G$ all sufficiently long commutators $[...[[g,x],x],\dots ,x]$ belong to ${\mathscr R}(g)$.  (Thus, $g$ is a right Engel element precisely when we can choose ${\mathscr R}(g)=\{ 1\}$.) We prove that if a finite group $G$ admits an automorphism $\f $ of prime order coprime to $|G|$ such that for some positive integer $m$ every element of the centralizer $C_G(\f )$ has a right Engel sink of cardinality at most $m$, then the index of the Fitting subgroup $F_1(G)$ is bounded in terms of $m$.
\end{abstract}
\maketitle

\section{Introduction}
\baselineskip=16pt

We use the left-normed simple commutator notation
$$[a_1,a_2,a_3,\dots ,a_r]:=[...[[a_1,a_2],a_3],\dots ,a_r]$$
and the abbreviation $[a,\,{}_kb]:=[a,b,b,\dots, b]$ where $b$ is repeated $k$ times.
Recall that an element $g$ of a  group $G$ is called a left Engel element if for every $x\in G$ the equation $[x,\,{}_{n} g]=1$ holds for some $n=n(x)$ depending on $x$. An element $g\in G$ is called a right Engel element if for every $x\in G$ the equation $[g,\,{}_{n} x]=1$ holds for some $n=n(x)$. An Engel group is defined as a group in which every element is a left Engel element (equivalently, every element is a right Engel element). Finite Engel groups are nilpotent by Zorn's theorem \cite[12.3.4]{rob}. Moreover, by Baer's theorem \cite[12.3.7]{rob}, a left Engel element of a finite group belongs to its Fitting subgroup, and a right Engel element of a finite group belongs to its hypercentre.  For many classes of groups theorems related to properties of Engel elements have been proved, achieving various structural properties related to nilpotency. For example, Wilson and Zelmanov \cite{wi-ze} proved that profinite Engel groups are locally nilpotent. At the same time, Golod's examples \cite{gol} show that Engel groups in general may not be locally nilpotent.

In recent papers \cite{khu-shu16,
khu-shu18,
khu-shu18a,
khu-shu19,
khu-shu20,
khu-shu-tra} we considered generalizations of Engel conditions for finite, profinite, and compact groups  using the concept of Engel sinks.

\begin{definition*} \label{dl}
 A \textit{left Engel sink} of an element $g$ of a group $G$ is a set ${\mathscr E}(g)$ such that for every $x\in G$ all sufficiently long commutators $[x,g,g,\dots ,g]$ belong to ${\mathscr E}(g)$, that is, for every $x\in G$ there is a positive integer $l(x,g)$ such that
 $[x,\,{}_{l}g]\in {\mathscr E}(g)$ for all $l\geq l(x,g).
 $
 \end{definition*}
 \noindent (Thus, $g$ is a left Engel element precisely when we can choose ${\mathscr E}(g)=\{ 1\}$, and $G$ is an Engel group when we can choose ${\mathscr E}(g)=\{ 1\}$ for all $g\in G$.)

 \begin{definition*} \label{dr}
 A \textit{right Engel sink} of an element $g$ of a group $G$ is a set ${\mathscr R}(g)$ such that for every $x\in G$ all sufficiently long commutators $[g,x,x,\dots ,x]$ belong to ${\mathscr R}(g)$, that is, for every $x\in G$ there is a positive integer $r(x,g)$ such that
 $[x,\,{}_{r}g]\in {\mathscr R}(g)$ for all $r\geq r(x,g).
 $
 \end{definition*}
 \noindent (Thus, $g$ is a right Engel element precisely when we can choose ${\mathscr R}(g)=\{ 1\}$, and $G$ is an Engel group when we can choose ${\mathscr R}(g)=\{ 1\}$ for all $g\in G$.)

In \cite{khu-shu18,khu-shu20} we considered finite, profinite, and compact (Hausdorff) groups in which  every element has a finite or countable left Engel sink and proved that then the group is close to being locally nilpotent in a certain precise sense. In \cite{khu-shu19,khu-shu20a} we obtained similar results concerning
right Engel sinks.

When $G$ is a finite group, then every element   has the \emph{smallest} left Engel sink, since the intersection of two left Engel sinks ${\mathscr E}'(g)$ and ${\mathscr E}''(g)$ is again a left Engel sink of $g$. Similarly,
every element  $g\in G$  has the \emph{smallest} right Engel sink. In this paper we shall always use the notation ${\mathscr E}(g)$ and ${\mathscr R}(g)$  to denote the smallest left and right Engel sinks of $g$, respectively, thus eliminating the ambiguity of this notation in the above definitions.

For finite groups we proved in  \cite{khu-shu18} and  \cite{khu-shu19} the following quantitative results.

\begin{theorem}[{\cite[Theorem~3.1]{khu-shu18}}]\label{t-f}
Let $G$ be a finite group, and $m$  a positive integer. Suppose that every element $g\in  G$ has a left Engel sink ${\mathscr E}(g)$ of cardinality  at most~$m$.
 Then $G$ has a normal subgroup $N$ of order bounded in terms of $m$ such that $G/N$ is nilpotent.
\end{theorem}

\begin{theorem}[{\cite[Theorem~3.1]{khu-shu19}}]\label{t-fr}
Let $G$ be a finite group, and $m$  a positive integer. Suppose that every element $g\in  G$ has a right Engel sink ${\mathscr R}(g)$ of cardinality  at most~$m$.
 Then $G$ has a normal subgroup $N$ of order bounded in terms of $m$ such that $G/N$ is nilpotent.
\end{theorem}

As is well known, often strong consequences for the structure of the group can be
derived from certain conditions imposed on fixed points of automorphisms.
 Some of the best-known examples include Thompson's theorem \cite{tho} on the nilpotency of a finite group with a fixed-point-free automorphism of prime order, as well as numerous other papers on finite groups admitting automorphisms with various restrictions on their fixed points.
This type of results in relation to automorphisms whose fixed points have restrictions on their Engel sinks were  recently obtained in \cite{acc-khu-shu, acc-shu-sil18, acc-shu-sil19, acc-sil18, acc-sil20, khu-shu202}.

The purpose of this paper is to obtain the following `automorphism extensions' of Theorems~\ref{t-f} and~\ref{t-fr}; one may also regard these extensions as results on `almost fixed-point-free' automorphisms.

\begin{theorem}\label{t-main}
Let $G$ be a finite group admitting an automorphism $\f$ of prime order coprime to $|G|$. Let $m$ be a positive integer such that every element $g\in C_G(\f)$ has a left Engel sink ${\mathscr E}(g)$ of cardinality at most $m$. Then $G$ has a metanilpotent normal subgroup of index bounded in terms of $m$ only. \end{theorem}

The conclusion for right Engel sinks is stronger. Note that while it is well-known that the inverse of a right Engel element is left Engel, there is no such a straightforward connection between left and right Engel sinks, and Theorem~\ref{t-mainr} is not a consequence of Theorem~\ref{t-main}.

\begin{theorem}\label{t-mainr}
Let $G$ be a finite group admitting an automorphism $\f$ of prime order coprime to $|G|$. Let $m$ be a positive integer such that every element $g\in C_G(\f)$ has a right Engel sink ${\mathscr R}(g)$ of cardinality at most $m$. Then $G$ has a nilpotent normal subgroup of index bounded in terms of $m$ only.
\end{theorem}

 Recall that the Fitting series starts with the Fitting subgroup $F_1(G)=F(G)$, and by induction, $F_{k+1}(G)$ is the inverse image of $F(G/F_k(G))$.
 Throughout the paper we shall write, say, ``$(a,b,\dots )$-bounded'' to abbreviate ``bounded above in terms of $a, b,\dots $ only''. Thus, the conclusion of Theorem~\ref{t-main} states that the index of $F_2(G)$ is $m$-bounded, and the conclusion of Theorem~\ref{t-mainr} states that the index of $F(G)$ is $m$-bounded.

 The results of Theorems~\ref{t-main} and \ref{t-mainr} are in a sense best-possible. There are easy examples showing that  in Theorem~\ref{t-main} one cannot obtain a bound in terms of $m$ (or even in terms of $|\f |$ and $m$) for the index of $F(G)$, nor can one prove that there is a normal subgroup of $(|\f|,m)$-bounded order with metanilpotent quotient.  In Theorem~\ref{t-main} one cannot prove that there is a normal subgroup of $(|\f |,m)$-bounded order with nilpotent quotient.

 We do not know whether the results of Theorems~\ref{t-main} and \ref{t-mainr} can be extended to an automorphism of prime order not coprime to $|G|$. Another possible direction of further studies would be considering automorphisms of composite order.

In \S\,\ref{s-red} we perform reduction of the  proofs  of both Theorems~\ref{t-main} and~\ref{t-mainr} to the case of soluble groups, using the classification of finite simple groups. Then in  \S\,\ref{s-l} the proof  of Theorem~\ref{t-main} about left Engel sinks is completed: first a `weak' upper bound for the index of $F_2(G)$  in terms of  $|\f|$ and $m$ is obtained, and then a `strong' upper bound, in which the dependence on $|\f|$ is eliminated. The proof of Theorem~\ref{t-mainr}  is completed in \S\,\ref{s-r}.
Requisite preliminary material is collected in~\S\,\ref{s-p}.

\section{Preliminaries}\label{s-p}

For a group $A$ acting by automorphisms on a group $B$ we use the usual notation for commutators $[b,a]=b^{-1}b^a$ and commutator subgroups $[B,A]=\langle [b,a]\mid b\in B,\;a\in A\rangle$, as well as for centralizers $C_B(A)=\{b\in B\mid b^a=b \text{ for all }a\in A\}$.
We usually denote by the same letter $\alpha$ the automorphism induced by an automorphism $\alpha $ on the quotient by a normal $\alpha$-invariant subgroup.

We say for short that an automorphism $\f$ of a finite group $G$ is a coprime automorphism if  the orders of $\f$ and $G$ are coprime: $(|G|,|\f|)=1$.  First we recall some properties of coprime automorphisms of finite groups. Several well-known facts about coprime automorphisms  will be often used without special references, but we recall them in the following lemma.

\begin{lemma}\label{l-nakr}
Let $\f$ be a coprime automorphism of a finite group $G$.
\begin{itemize}
  \item[\rm (a)] For every prime $p$ dividing $|G|$, there is a $\f$-invariant Sylow $p$-subgroup of $G$.
  \item[\rm (b)] If $G$ is soluble, then for every set $\pi$ of primes dividing $|G|$, there is a $\f$-invariant Hall $\pi$-subgroup of $G$.
  \item[\rm (c)] If $N$ is a normal $\f$-invariant subgroup of~$G$, then the fixed points of $\f$ in  the quotient $G/N$ are covered by the fixed points of $\f$ in $G$, that is, $C_{G/N}(\f ) = C_G(\f)N/N$.
\end{itemize}
  \end{lemma}

In the next useful lemma the condition that the group is nilpotent cannot be dropped.

 \begin{lemma}[{\cite[Lemma~2.4]{rod-shu}}]\label{l-copr}
Let $\varphi$ be a coprime automorphism of a finite nilpotent
group $G$. Then any element $g\in G$ can be uniquely written in the form
$g=cu$, where $c\in C_G(\varphi)$ and $u=[v,\varphi ]$ for some $v\in G$.
 \end{lemma}

As a special case, we have the following well-known fact.

\begin{lemma}\label{l-ab}
Let $\f$ be a coprime automorphism of an abelian finite group $G$. Then
$G=[G,\f]\times C_G(\f)$.
 \end{lemma}

The following lemma appeared in \cite{khu-shu16}  in a slightly less general form, so we reproduce the proof for the benefit of the reader.

\begin{lemma}\label{l33}
Let $V$ be an abelian  finite group, and $U$ a group of coprime
automorphisms of $V$.  If $|[V,u]|\leq m$ for every $u\in U$, then $|[V,U]|$ is $m$-bounded, and therefore $|U|$ is also $m$-bounded.
\end{lemma}

\begin{proof}
First suppose that $U$ is abelian. Pick $u_1\in U$ such that $[V,u_1]\ne 0$. By Lemma~\ref{l-ab}, $V= [V,u_1]\times C_V(u_1)$, and both summands are $U$-invariant, since $U$ is abelian. If $C_U([V,u_1])=1$, then $|U|$ is $m$-bounded and $[V,U]$ has $m$-bounded order being generated by $[V,u]$, $u\in U$. Otherwise pick $1\ne u_2\in C_U([V,u_1])$; then $V= [V,u_1] \times [V,u_2] \times C_V(\langle u_1,u_2\rangle )$. If $1\ne u_3\in C_U([V,u_1]\times [V,u_2])$, then $V= [V,u_1]\times [V,u_2]\times [V,u_3] \times C_V(\langle u_1,u_2,u_3\rangle )$, and so on. If $C_U([V,u_1]\times\dots \times [V,u_k])=1$ at some $m$-bounded step $k$, then again $[V,U]$ has $m$-bounded order. However, if there are too many steps, then
for the element $w=u_1u_2\cdots u_k$ we shall have $0\ne [V,u_i]= [[V,u_i],w]$, so that $[V,w] = [V,u_1]\times\dots \times [V,u_k]$ will have order greater than $m$, a contradiction.

We now consider the general case. Since every element $u\in U$ acts faithfully on $[V,u]$, the exponent of $U$ is $m$-bounded. If $P$ is a Sylow $p$-subgroup of $U$, let $M$ be a maximal normal abelian subgroup of $P$. By the above, $|[V,M]|$ is $m$-bounded. Since $M$ acts faithfully on $[V,M]$, we obtain that $|M|$ is $m$-bounded. Hence $|P|$ is $m$-bounded, since $C_P(M)= M$ and $P/M$ embeds in the automorphism group of $M$.  Since $|U|$ has only $m$-boundedly many prime divisors, it follows that $|U|$ is $m$-bounded. Since $[V,U]=\prod_{u\in U}[V,u]$, we obtain that $|[V,U]|$ is also $m$-bounded.
\end{proof}

Along with the aforementioned celebrated theorem of Thompson on the nilpotency of finite groups with a fixed-point-free automorphism of prime order, we shall also use the following theorem of Hartley and Meixner about   an `almost fixed-point-free' automorphism.

\begin{theorem}[{\cite{h-m}}]\label{t-h-m}
  If a finite soluble group $G$ admits an automorphism $\f$ of prime order $p$ such that $|C_G(\f)|=s$, then the index of $F(G)$ is $(p,s)$-bounded.
\end{theorem}

We shall use the following lemma on actions of Frobenius groups.

\begin{lemma}[{\cite[Lemma~2.4]{kms}}]\label{l-frob}
Suppose that a finite group $G$ admits a Frobenius group of automorphisms
$FH$ with kernel $F$ and complement $H$ such that $C_G(F)=1$. Then
 $G=\langle C_G(H)^f\mid f\in F\rangle$.
\end{lemma}

We denote by  $\gamma _{\infty}(G)=\bigcap _i\gamma _i(G)$ the intersection of the lower central series of a group~$G$; when $G$ is finite, then $G/\gamma _{\infty}(G)$ is the largest nilpotent quotient of $G$. We note the following elementary fact.

\begin{lemma}\label{l-gf}
If $|\g_\infty(G)|=n$, then the centralizer $C_G(\g_\infty(G))$ is a nilpotent normal subgroup of  $n$-bounded index, so that $|G/F(G)|$ is $n$-bounded.
\end{lemma}

\begin{proof}
  Indeed, for any $c_i\in  C_G(\g_\infty(G))$ a long enough commutator $[c_1,\dots ,c_k]$ belongs to $\g_\infty(G)$ and then $[c_1,\dots ,c_k,c_{k+1}]=1$. Thus, $C_G(\g_\infty(G))\leq F(G)$. The quotient $G/C_G(\g_\infty(G))$ embeds  into the automorphism group of $\g_\infty(G)$ and therefore also has $n$-bounded order. Hence $|G/F(G)|$ is also $n$-bounded.
\end{proof}

We now recall some elementary properties of Engel sinks.  In cases where we need to consider the left or right Engel sink constructed with respect to a subgroup $H$ containing~$g$, we write ${\mathscr E}_H(g)$ or ${\mathscr R}_H(g)$, respectively. The following lemma is a consequence of the properties of coprime actions and a formula in Heineken's paper \cite{hei} on a connection between left and right Engel commutators.

\begin{lemma}[{\cite[Lemma~3.2]{khu-shu19}}]\label{l-sink}
 If $V$ is an abelian subgroup of a finite group $G$, and $g\in G$ an element normalizing $V$  such that $(|V|,|g|)=1$, then
 $$[V,g]={\mathscr E}_{V\langle g\rangle}(g)={\mathscr R}_{V\langle g\rangle}(g).$$
\end{lemma}

\begin{remark}\label{r-inh}
In view of Lemma~\ref{l-nakr}(c), the hypotheses of Theorems~\ref{t-main} and~\ref{t-mainr}  are inherited by any  $\f$-invariant section, that is, a quotient $A/B$ of a  $\f$-invariant  subgroup $A$ by a  $\f$-invariant subgroup $B$  normal in $A$. We shall freely use this fact without special references.
\end{remark}

\section{Bounding the index of the soluble radical}\label{s-red}

In this section we perform a reduction of the proof of Theorems~\ref{t-main} and~\ref{t-mainr} to the case of soluble groups. For that we use a recent result of Guralnick and Tracey \cite{gur-tra}.

\begin{theorem}[{\cite[Corollary~1.9]{gur-tra}}]\label{t-gt}
 Let $\alpha\in {\rm Aut}\,G$ be an involutive automorphism of a finite group $G$, and let
$J(\alpha)=\{g\in G\mid \text{$g$ has odd order and } g^{\alpha}=g^{-1}\}$.  If $G=[G,\alpha ]$, then the index of the Fitting subgroup $F(G)$ is bounded in terms of $|J(\alpha )|$, namely,
$|G/F(G)|\leq |J(\alpha )|!^4$.
\end{theorem}

In fact, we can derive from Theorem~\ref{t-gt} a general proposition about groups of coprime automorphisms whose fixed points have bounded Engel sinks.
We state this proposition here in greater generality than needed in the present paper, as it may find applications in other studies.

\begin{prop}\label{pr}
 Let $G$ be a finite group admitting a group of coprime automorphisms~$H$. Suppose that $m$ is a positive integer  such that every $2$-element in $C_G(H)$ has a left (or right) Engel sink  of cardinality at most~$m$. Then the soluble radical $S(G)$ of $G$ has $m$-bounded index.
\end{prop}

(In the statement,  ``left (or right)'' is applied individually to the $2$-elements of $C_G(H)$, so that it may be different, left or right, Engel  sinks for different  $2$-elements of $C_G(H)$.)

\begin{proof}
  Since the hypothesis is inherited by $G/S(G)$, we can assume that $S(G)=1$. Then the generalized Fitting subgroup
  \begin{equation}\label{e-gt}
    F^*(G)=S_1\times \dots \times S_n
  \end{equation}
  is a direct product of non-abelian finite simple groups  $S_i$, which are permuted by  $H$.
  Since  the centralizer of $F^*(G)$ is trivial and $G$ embeds into the automorphism group of $F^*(G)$, it is sufficient to  obtain a bound for $|F^*(G)|$ in terms of $m$. We can simply assume that $G=F^*(G)$.

  As proved by Wang and Chen \cite{wan-che} on the basis of the classification of finite simple groups, a finite group admitting a group of coprime automorphisms with fixed-point subgroup of odd order is soluble. Therefore $C_G(H)$ contains involutions. Moreover, we can choose an involution $\alpha\in C_G(H)$ with non-trivial projections onto each of the factors $S_i$ in \eqref{e-gt}.   Indeed, let  $\{S_{i_1},\dots,S_{i_k}\}$ be one of the orbits of $H$ in its permutational action on the set $\{S_1,\dots,S_n\}$.   Clearly, any non-trivial element of the  product $S_{i_1}\times \cdots\times S_{i_k}$ centralized by $H$ must have non-trivial projections onto each of the factors $S_{i_j}$. There is an involution in the centralizer of $H$ in the product $S_{i_1}\times \cdots\times S_{i_k}$. We now take $\alpha\in C_G(H)$ to be the product of these involutions over all orbits of $H$. Then $\alpha\in C_G(H)$ is an involution with non-trivial projections onto each of the factors $S_{i}$ in \eqref{e-gt}. It follows that $G=F^*(G)=[G,\alpha ]$.

  If $g\in J(\alpha )$, that is, $g^{\alpha }=g^{-1}$, then $[\langle g\rangle,\alpha ]=\langle g\rangle$, since $g$ has odd order. Then $\langle g\rangle\subseteq \mathscr R (\alpha)\cap \mathscr E(\alpha)$ by Lemma~\ref{l-copr}. Therefore, $J(\alpha )\subseteq \mathscr R (\alpha)\cap \mathscr E(\alpha)$, so that $|J(\alpha )|\leq m$. We also have $[G,\alpha ]=G$ and $F(G)=1$, and Theorem~\ref{t-gt} completes the proof.
\end{proof}

In our situation, as an obvious consequence we obtain the following.

\begin{corollary}
\label{l-srs}
Let $G$ be a finite group admitting  a coprime automorphism $\f$ of prime order. Suppose that $m$ is a positive integer such that one of the following conditions holds (or both):
\begin{itemize}
  \item[\rm (a)] every element $g\in C_G(\f)$ has a left Engel sink ${\mathscr E}(g)$ of cardinality at most $m$, or
  \item[\rm (b)] every element $g\in C_G(\f)$ has a right Engel sink ${\mathscr R}(g)$ of cardinality at most $m$.
\end{itemize}
  Then the soluble radical  $S(G)$ has $m$-bounded index in $G$.
\end{corollary}

\section{Completion of the proof for left Engel sinks}\label{s-l}

\begin{proof}[Proof of Theorem~\ref{t-main}] Recall that $G$ is a finite group admitting  an automorphism $\f$ of prime order  $p$ coprime to $|G|$, and every element of $C_G(\f)$ has a left Engel sink of cardinality at most $m$; we need to prove that  $|G/F_2(G)|$ is $m$-bounded. By Corollary~\ref{l-srs}(a) the index of the soluble radical $|G/S(G)|$ is $m$-bounded. Therefore we can assume that $G$ is soluble.

First we obtain a `weak' upper bound for $|G/F_2(G)|$ in terms of $p$ and $m$.

\begin{lemma}\label{l-sol}
The index of $F_2(G)$ is $(p,m)$-bounded. \end{lemma}

\begin{proof}
By Gasch\"utz's theorem the image of the Fitting subgroup $F(G)$ in the quotient  $G/\Phi (G)$  of $G$ by the Frattini subgroup $\Phi (G)$ is the Fitting subgroup of this quotient (see \cite[5.2.15]{rob}),
and hence the same holds for all terms of the Fitting series. Therefore, since the hypothesis is inherited by $G/\Phi (G)$, we can assume from the outset that $\Phi (G)=1$. Then $F(G)$ is abelian by Gasch\"utz's theorem.

To lighten the notation, let $F_i=F_i(G)$. Since $F_1$ is abelian, every element $g\in
F_1C_{F_2}(\f)$  has Engel sink of cardinality at most $m$. Indeed, let $g=uc$, where  $u\in F_1$ and $c\in C_{F_2}(\f)$. For any $h\in G$, a long enough commutator $d=[h,\,{}_kg]$ belongs to $F_1$, since $F_2/F_1$ is nilpotent. Then $[d,\,{}_ng]=[d,\,{}_nuc]=[d,\,{}_nc]$ for any $n$, since $F_1$ is abelian. As  a result,  ${\mathscr E}(g)$ is contained in ${\mathscr E}(c)$, which has cardinality at most $m$ by hypothesis.

  Applying Theorem~\ref{t-f}, we now obtain that $\g _{\infty}(F_1C_{F_2}(\f))$  has $m$-bounded order. It follows that the Fitting subgroup  $F(F_1C_{F_2}(\f))$ has $m$-bounded index in $F_1C_{F_2}(\f)$
 by Lemma~\ref{l-gf}.
  But the Fitting subgroup  $F(F_1C_{F_2}(\f))$  is equal to $F_1$, since $F_1C_{F_2}(\f)$ is a subnormal subgroup of $G$. Hence $C_{F_2/F_1}(\f) =C_{F_2}(\f)F_1/F_1$ has $m$-bounded order.

  We can apply the same arguments to the quotient of $G/F_1$ by its Frattini subgroup. We obtain that $C_{F_3/F_2}(\f)$ also has $m$-bounded order. As a result,  $C_{F_3/F_1}(\f)$ has $m$-bounded order. By Theorem~\ref{t-h-m} then the Fitting subgroup $F(F_3/F_1)$ has  $(p,m)$-bounded  index in $F_3/F_1$. But, obviously,  $F(F_3/F_1)=F_2/F_1$, so that $F_3/F_2$ has $(p,m)$-bounded order. Since $G$ is soluble,  $F_3/F_2$ contains its centralizer in $G/F_2$. Then the group $G/F_2$ embeds into the automorphism group of $F_3/F_2$ and therefore also has $(p,m)$-bounded order.
\end{proof}

We proceed with the proof of Theorem~\ref{t-main}, where we need to obtain a `strong' bound, in terms of $m$ only, for $|G/F_2(G)|$. If $p\leq m$, then the bound for $|G/F_2(G)|$ in terms of $p$ and $m$ obtained in Lemma~\ref{l-sol} is a required bound in terms of $m$ only. Therefore we can assume that $p>m$.
 We observe that when $x\in C_G(\f)$, the Engel sink $ \mathscr E(x)$ is $\f$-invariant. Then the  condition  $p>m\geq |\mathscr E(x)|$ implies that
    \begin{equation}\label{e0}
    \mathscr E(x)\subseteq C_G(\f )\qquad\text{for every $x\in C_G(\f)$}.
  \end{equation}

  Let $N=\langle C_G(\f )^G\rangle$ be the normal closure of $C_G(\f )$. Our goal now is to prove that $\gamma_{\infty}(N)$ has $m$-bounded order. The bulk of the proof is in the following key lemma.

 \begin{lemma}\label{l-gnc}
    We have $\gamma_{\infty}(N) \leq C_G(\f )$.
 \end{lemma}

 \begin{proof}
 We use induction on the order of $G$. To lighten the notation, we write $\g =\gamma_{\infty}(N)$.

 Recall that $G$ is soluble. Let $V$ be a minimal normal subgroup of $G\langle\f\rangle$, so it is an elementary abelian $q$-group for some prime $q$. By induction applied to $G/V$, we have $\gamma \leq VC_G(\f )$. By the minimality of $V$, either $\g \cap V=1$ or $\g \geq V$. If $\g \cap V=1$, then $[\g, \f]\leq [VC_G(\f ), \f]=[V,\f]\cap \g=1$, whence $\g\leq C_G(\f)$ and the proof is complete.  Hence we can assume that $\g \geq V$.

 We have $V=[V,\f ]\times C_V(\f)$ by Lemma~\ref{l-ab}. The subgroup $[V,\f]$ is $C_G(\f)$-invariant. Hence, $\g =[V,\f](C_G(\f)\cap \g)$. Any $q'$-element $x$ of $C_G(\f)$ acts trivially on $[V,\f]$ by Lemma~\ref{l-sink} because  $\mathscr E(x)\subseteq C_G(\f)$ by \eqref{e0}. We record this property in the form
 \begin{equation}\label{e1}
   [\g,x]\leq C_G(\f)\qquad \text{for any $q'$-element $x$ of $C_G(\f)$.}
 \end{equation}

 Any Hall $q'$-subgroup of $\g$ has the form $H^v$, where $v\in V$ and $H$ is a Hall $q'$-subgroup of $C_G(\f)\cap \g$. Let $v=v_1v_2$ for $v_1\in [V,\f]$ and $v_2\in C_V(\f)$. Since $[H, v_1]=1$ by \eqref{e1},  we have $H^v=H^{v_2}\leq C_G(\f)\cap \g$. As a result, the subgroup $O^q(\g)$ generated by all its Hall $q'$-subgroups is contained in $C_G(\f)\cap \g$. If $O^q(\g)\ne 1$, then by induction $\g\leq O^q(\g)C_G(\f)\leq C_G(\f)$ and the proof is complete. Therefore we can assume that  $O^q(\g)= 1$, that is,  $\g$ is a $q$-group.

 If $\Phi (\g)\ne 1$, then $\g\leq \Phi (\g) C_V(\f)$ by induction. But
 $$\Phi (\g)=\big[[V,\f ], C_G(\f)\cap \g\big]\cdot \Phi (C_G(\f)\cap \g),$$
 so then $[V,\f ]\leq \big[[V,\f ], C_G(\f)\cap \g\big]$, where the right-hand side is strictly smaller than $[V,\f ]$, since $\g$ is nilpotent. This is a contradiction, unless $[V,\f ]=1$, when the proof is complete. Hence we can assume that $\Phi (\g)= 1$, so that $\g$ is an elementary abelian $q$-group.

 Let $R$ be a $\f$-invariant  Hall $q'$-subgroup of $N$. Then $[\g, R] =\g$ and therefore, by Lemma~\ref{l-ab},
 \begin{equation}\label{e2}
   C_{\g}(R)=1,
 \end{equation}
 since $\g$ is abelian.

 We now consider $\g$ as a right  $\F_qG\langle\f\rangle$-module and extend the ground field to a finite splitting field $k$ of  $G\langle\f\rangle$; let $M$ be the resulting  $kG\langle\f\rangle$-module. The additive group of $M$ is a finite $q$-group, and therefore both $R$ and $\f$ act by coprime automorphisms, to which Lemma~\ref{l-nakr}(c) on covering fixed points applies. Note that by \eqref{e2} we have
 \begin{equation}\label{e3}
C_M(R)=0.
 \end{equation}
 We continue using the same notation for centralizers and commutator subgroups, albeit using the additive structure of $M$.

 We now consider an unrefinable series of $kG\langle\f\rangle$-submodules connecting $0$ with  $M$. We consider  an arbitrary factor $U$ of this series, which is an irreducible $kG\langle\f\rangle$-module. Let the bar denote  the images of elements and subgroups in the action on $U$. Note that $\overline{\g}=1$, so that $\bar N$ is nilpotent. We can sometimes drop the bar when considering the action of $G\langle\f\rangle$ on $U$; for example,
     \begin{equation}\label{e4}
C_U(R)=0
 \end{equation}
 by \eqref{e3} and Lemma~\ref{l-nakr}. We also have
   \begin{equation}\label{e5}
[U,C_R(\f)]\leq C_U(\f)
 \end{equation}
  by \eqref{e1}.

  Let $U=W_1\oplus\dots\oplus W_n$ be the decomposition of $U$ into the sum of Wedderburn homogeneous  components with respect to $\bar R$, which is a normal subgroup of $\bar G$, since $\bar N$ is nilpotent.  By Clifford's theorem, the group $ G\langle \f\rangle$ transitively permutes the components $W_i$, and the kernel of this permutational action contains $ RC_{ G}( R)\geq  N$.

  We claim that  $C_R(\f )$ acts non-trivially on one of the components $W_i$. Indeed, otherwise $C_R(\f )$ acts trivially on $U$. But $\bar R$  is the normal closure of $C_{\bar R}(\f )$ because $\bar N$ is nilpotent and is the normal closure of $C_{\bar G}(\f)$. Thus we obtain that $R$ acts trivially on $U$, which contradicts \eqref{e4}.

  For definiteness, let $W_1$ be a component on which  $C_R(\f )$ acts non-trivially. By \eqref{e5} then
$[W_1, C_R(\f )]$ (which is contained in $W_1$)  is a non-trivial subspace of $C_U(\f )$. Hence   $\f$ belongs to the stabilizer of $W_1$. Since $\bar N \leq \bar RC_{\bar G}(\bar R)$ is in the kernel of the permutational action on the set of the $W_i$ and $\f$ acts fixed-point-freely on $G/N$, all orbits of $\f$ on the set of the $W_i$ have length $p$ except the one-element orbit $\{W_1\}$. Indeed, let $S$ be the stabilizer of $W_1$. If $g\not\in S$ and $W_1g\f = W_1g$, then, since $\f\in S$, the coset $Sg$ is  $\f$-invariant. Since $|Sg|$ is coprime to $p=|\f|$, then $\f$ would have a fixed point in $Sg$, contrary to $C_G(\f)\leq N\leq S$.

We claim that $\f$ actually cannot have any orbits of length $p$ on the set of the $W_i$, so that  $\{W_1\}$ is the only orbit of $\f$, which of course means that $U$ is a homogeneous $k\bar R$-module.
Suppose the opposite, and let  $W_1t, W_1{t^{\f }}, \dots  , W_1{t^{\f ^{p-1}}}$ be a regular orbit of $\f$  (here, $W_1{t^{\f ^{i}}}=W_1{t{\f ^{i}}}$, since $\f\in S$). Then $C_R(\f )$ must centralize all these $W_1{t^{\f ^{i}}}$ because of \eqref{e5}:
\begin{equation}\label{e6}
 [W_1{t^{\f ^{i}}}, C_R(\f )]=0 \qquad \text{for } i=0,1,\dots ,p-1.
\end{equation}
For every $i=0,1,\dots ,p-1$, the subgroup
 $C_R(\f )^{t^{\f ^{i}}}$ acts  nontrivially  on $W_1{t^{\f ^{i}}}$ and centralizes all the other $W_1{t^{\f ^{s}}}$ for $s\ne i$. Let tilde denote the images in the action on the direct sum
 $$W_1t\oplus W_1{t^{\f }}\oplus  \dots  \oplus W_1{t^{\f ^{p-1}}}.$$
 Then the
subgroups $\widetilde{C_R(\f )^{t}}, \widetilde{C_R(\f )^{t^\f }}, \dots  , \widetilde{C_R(\f )^{t^{\f ^{p-1}}}}$ commute and generate a direct product
$$\widetilde{C_R(\f )^{t}}\times  \widetilde{C_R(\f )^{t^\f }}\times  \dots  \times  \widetilde{C_R(\f )^{t^{\f ^{p-1}}}},$$
in which the factors are permuted by~$\f$. The diagonal of this direct product is an image of a subgroup
of $C_R(\f )$ and acts non-trivially on these $W_1^t, W_1^{t^{\f }}, \dots  , W_1^{t^{\f ^{p-1}}}$, contrary to~\eqref{e6}. This contradiction shows that $\f$ has no orbits of length $p$.

 Thus, $\{W_1\}$ is the only orbit of $\f$, so that $U=W_1$ is a homogeneous $k\bar R$-module. Since $\bar R$ is nilpotent and $k$ is a splitting field, the centre $Z(\bar R)$ acts by scalar multiplications on~$W_1$. Hence $Z(\bar R)$ commutes with~$\bar \f$. By Lemma~\ref{l-nakr}(c) there is an element  $z\in C_R(\f)$ whose image is a non-trivial element of $Z(\bar R)$. Since  $[U,z]\leq C_U(\f )$ by \eqref{e5} and $U=[U,z]$, we obtain that $[U, \f ]=0$.

 Since the above arguments apply to every irreducible factor of that series in $M$, we obtain by Lemma~\ref{l-nakr}(c) that $[M,\f]=0$, which means that $\g\leq C_G(\f)$, as required.
 \end{proof}

 \begin{lemma}\label{l-gn}
 The order of $\g_{\infty}(N)$ is $m$-bounded.
 \end{lemma}

 \begin{proof}
   By Lemma~\ref{l-copr} every element of  $N/\g_{\infty}(N)$ can be written as a product of a commutator $[h,\f ]$ for $h\in N/\g_{\infty}(N)$ and an element from $C_{N/\g_{\infty}(N)}(\f )=C_{N}(\f )\g_{\infty}(N)/\g_{\infty}(N)$. Since  $\g_{\infty}(N)\leq C_{N}(\f )$ by Lemma~\ref{l-gnc}, it follows that every element    $n\in N$ can be written as a product $n=[g,\f ]c$ for $g\in N$ and $c\in C_{N}(\f )$.

We now claim that $\mathscr E(n)=\mathscr E(c)$. This follows from the fact that $[g,\f]$ centralizers $\g_{\infty}(N)$. Indeed, since $\g_{\infty}(N)\leq C_{N}(\f )$ by Lemma~\ref{l-gnc} and $\g_{\infty}(N)$ is a normal subgroup of $G\langle\f\rangle$, the centralizer $C_G(\g_{\infty}(N))$ is a normal subgroup that contains $\f$ and therefore also contains $[g,\f ]$.

The quotient $G/N$ admits a fixed-point-free automorphism $\f$ of prime order and therefore  is nilpotent by Thompson's theorem~\cite{tho}. Therefore for any $x\in G$ we have $[x,\,{}_sn]\in N$ for some $s$, and then $[x,\,{}_{s+t}n]\in \g_{\infty}(N)$ for some $t$. Since $[g,\f]$ centralizers $\g_{\infty}(N)$ as shown above, we further have  $[x,\,{}_{s+t+u}n]=[[x,\,{}_{s+t}n],\,{}_uc]$ and this element belongs to $\mathscr E(c)$ for large enough~$u$.

As a result, all elements of $N$ have Engel sinks of size at most $m$ and therefore $\g_{\infty}(N)$ has $m$-bounded order by Theorem~\ref{t-f}.
\end{proof}

 We now finish the proof of Theorem~\ref{t-main}. In the remaining case $p>m$ we have $\g_{\infty}(N)$ of $m$-bounded order by Lemma~\ref{l-gn}. Then $C_G(\g_{\infty}(N))$ has $m$-bounded index and is metanilpotent, since $C_G(\g_{\infty}(N))/(N\cap C_G(\g_{\infty}(N))$ is nilpotent and $N\cap C_G(\g_{\infty}(N))$ is nilpotent by Lemma~\ref{l-gf}.
  \end{proof}

\section{Completion of the proof for right Engel sinks}\label{s-r}

\begin{proof}[Proof of Theorem~\ref{t-mainr}] Recall that $G$ is a finite group admitting  an automorphism $\f$ of prime order  $p$ coprime to $|G|$, and every element of $C_G(\f)$ has a right Engel sink of cardinality at most $m$; we need to prove that  $|G/F(G)|$ is $m$-bounded. By Corollary~\ref{l-srs}(b) the index of the soluble radical $|G/S(G)|$ is $m$-bounded. Therefore we can assume that $G$ is soluble.

Since $G/F_2(G)$ embeds in the automorphism group of $F_2(G)/F(G)$, we can assume that $G=F_2(G)$, so that $G/F$ is nilpotent. We choose Thompson's critical subgroup $C_p$ in each Sylow $p$-subgroup $P$ of $G/F$, which is a characteristic subgroup of $P$ such that  $C/Z(C)$ is an elementary
abelian $p$-group and $C_P(C_p)=Z(C)$ (see  \cite[Theorem~5.3.11]{gor}). Then the product $C=\prod_pC_p$ is a characteristic subgroup of $G/F$ such that $C/Z(C)$ is a direct product of elementary
abelian groups and $C_{G/F}(C)=Z(C)$. Since $(G/F)/Z(C)$ embeds into the automorphism group of $C$, it is sufficient to prove that $|C|$ is $m$-bounded. Therefore we can replace $G$ by the inverse image of $C$ and assume that $G/F=C$.

Since $F(G)/Z(G)=F(G/Z(G)$ we can assume that $Z(G)=1$.

Consider the Fitting subgroup $F(G\langle\f\rangle)$ of the semidirect product $G\langle\f\rangle$. We have $F(G\langle\f\rangle)\cap G=F(G)$, so it suffices to show that the index of $F(G\langle\f\rangle)$ in $G\langle\f\rangle$ is $m$-bounded. By Gasch\"utz's theorem  \cite[Satz~III.4.2]{hup} the image of the Fitting subgroup in the quotient  of $F(G\langle\f\rangle)$ by its Frattini subgroup is the Fitting subgroup of this quotient. Since the hypothesis is inherited by this quotient, we can assume that  the Frattini subgroup of  $F(G\langle\f\rangle)$ is trivial. Then $F(G\langle\f\rangle)$ is a direct product of minimal normal subgroups of $G\langle\f\rangle$ by  Gasch\"utz's theorem  \cite[Satz~III.4.5]{hup}. Therefore $F(G)$
is a direct product of minimal normal $\f$-invariant subgroups, which
are elementary abelian $q$-groups for various~$q$. To lighten the notation,  in what follows  we write $F=F(G)$.

\begin{lemma}\label{l-gfmb}
The centralizer $C_{G/F}(\f)$ of $\f$ in $G/F$ has  $m$-bounded order.
\end{lemma}

\begin{proof} Let $q$ be any prime dividing $|C_{G/F}(\f)|$, and let $Q$ be a Sylow $q$-subgroup of $C_{G/F}(\f)$. Consider any non-trivial element $g\in Q$, and let $\hat g$ be a $q$-element of $G$  that is a pre-image of $g$ chosen in $C_G(\f)$ in accordance with Lemma~\ref{l-copr}. Then $\hat g$ induces by conjugation a non-trivial coprime automorphism of the Hall $q'$ subgroup $F_{q'}$ of $F$. Indeed, if  $\hat g$ centralized $F_{q'}$, then $\hat g$ would centralize all factors of a principal series of $G$ and would belong to $F$ by \cite[5.2.9]{rob}. We have $[F_{q'},\hat g]\subseteq \mathscr R(\hat g)$ by Lemma~\ref{l-sink}. Hence $[F_{q'},\hat g]=[F_{q'},  g]$ has $m$-bounded order for any $g\in Q$. Then $|Q|$ is $m$-bounded by Lemma~\ref{l33}. In particular, the prime $q$ is bounded above in terms of $m$. As a result,  $|C_{G/F}(\f)|$ is $m$-bounded.
\end{proof}

Due to Lemma~\ref{l-gfmb} we can use induction on $|C_{G/F}(\f)|$ in the proof of Theorem~\ref{t-mainr} to show that we can assume that $C_{G/F}(\f)=1$. Namely, if $C_{G/F}(\f)\ne 1$, then for some prime $q$ and a Sylow $q$-subgroup $Q$ of $G/F$, there is a non-trivial element  $c\in C_{Q}(\f)$. As before, $[F_{q'},c]\ne 1$. If $c\not\in Z(Q)$, then, since $Q/Z(Q)$ is elementary abelian, by Maschke's theorem there is a normal $\f$-invariant subgroup $Q_1$ of $Q$ of index $q$ not containing $c$. Since here $q$ is $m$-bounded, the inverse image $G_1$  of the product of $Q_1$ and the Hall $q'$-subgroup of $G/F$ is a $\f$-invariant subgroup of $m$-bounded index with $F(G_1)=F$ such that $|C_{G_1/F}(\f)|<|C_{G/F}(\f)|$. Then the induction hypothesis applied to $G_1$ completes the proof. If, however,  $c\in Z(Q)$, then let  $\hat c$ be a $q$-element of $G$  that is a pre-image of $c$ chosen in $C_G(\f)$ in accordance with Lemma~\ref{l-copr}. Then $[F_{q'},c]$ is a normal $\f$-invariant subgroup of $G$, since $c\in Z(G/F)$ and $F$ is nilpotent. Furthermore,  $[F_{q'},c]\subseteq \mathscr R(\hat c)$ by Lemma~\ref{l-sink}, and therefore $[F_{q'},c]$ has  $m$-bounded order. We obtain that $C_G([F_{q'},c])$ is a $\f$-invariant subgroup of $m$-bounded index, which does  not contain $c$, since $[[F_{q'},c],c]=[F_{q'},c]\ne 1$.  Then the induction hypothesis applied to $C_G([F_{q'},c])$ completes the proof.

Thus, we can assume that $C_{G/F}(\f)=1$ for the rest of the proof. This means that $(G/F)\langle\f\rangle$ is a Frobenius group with kernel $G/F$ and complement $\langle\f\rangle$ acting on $F$. It follows by Lemma~\ref{l-frob} that every $\f$-invariant  normal section $S$ of $G$ such that $C_S(G/F)=1$ is the normal closure of $C_S(\f)$ under the action of $G$:
\begin{equation}\label{e-clos}
  S=\langle C_S(\f)^g\mid g\in G\rangle\qquad \text{if }C_S(G/F)=1.
\end{equation}

Recall that  $F$ is a direct product of minimal normal $\f$-invariant subgroups (which are elementary abelian $q$-groups for various, not necessarily different, primes $q$):
\begin{equation}\label{e-prod}
F=V_1\times \dots\times V_k.
\end{equation}
Let  $V$ be one of these factors $V_i$, and let $\bar  G=G/C_G(V)$ be the image of $G$ in the action by conjugation on $V$. We record some properties of the action of $G$ on $V$ in the following lemma.

\begin{lemma}\label{e-ch0}
If $V$ is an elementary abelian $q$-group, then $ \bar G$ is a non-trivial $q'$-group and $C_V(G)=1$.
\end{lemma}

\begin{proof}
Indeed, suppose that $Q$ is a nontrivial Sylow $q$-subgroup of $\bar G$.   Since $\bar G$ is nilpotent, then $Z(VQ)\cap V$ is a proper $G$-invariant subgroup of  $V$,  which is  also $\f$-invariant. This contradicts the minimality of $V$. Since $Z(G)=1$ by our assumption, we have $\bar G\ne 1$. Since $C_V(G)$ is normal in $G\langle\f\rangle$, we must have $C_V(G)=1$ by the minimality of $V$.
\end{proof}

We now prove a key lemma in the proof of Theorem~\ref{t-mainr}.

\begin{lemma}\label{l-mainr}
For any factor $V=V_i$ in the product \eqref{e-prod}, the order of $\bar G=G/C_G(V)$ is $m$-bounded.
\end{lemma}

\begin{proof}
Let $V$ be an elementary abelian $q$-group. Since  $\bar G\ne 1$, we have  $V= \langle C_V(\f)^g\mid g\in G\rangle$ by \eqref{e-clos}.  Let
 $1\ne c\in C_V(\f)$.  For any $x\in \bar G$ and any positive integer $n$ we have $[c,\,{}_{q^n}x]= [c,x^{q^n}]$. Indeed, regarding $V$ as an ${\Bbb F}_q\bar G$-module, we see that
 $$[c,\,{}_{q^n}x]= c(x-1)^{q^n}=c(x^{q^n}-1)=[c,x^{q^n}],$$
 in characteristic $q$. For all large enough $n$ the element $[c,\,{}_{q^n}x]$ belongs to the right Engel sink $\mathscr R(c)$. Choosing $n$ to be the same large enough integer for all elements  $x\in \bar G$, we obtain that $[c,x^{q^n}]$ takes at most $m$ values when $x$ runs over $\bar G$, since $|\mathscr R(c)|\leq m$. Since $\bar G$ is a finite $q'$-group by Lemma~\ref{e-ch0}, the elements $x^{q^n}$  run over the entire  $\bar G$ as $x$ does.  Thus,  $[c,x]$ takes at most $m$ values when $x$ runs over $\bar G$, which  means that the index of the centralizer $C_{\bar G}(c)$ in $\bar G$ is at most $m$. The intersection $Z=\bigcap_{g\in \bar G} C_{\bar G}(c)^g$ is a $G$-invariant subgroup of $m$-bounded index, and this intersection is also $\f$-invariant, since $C_{\bar G}(c)$ is. Then $[c^g,Z]=0$ for all $g\in G$. But $V= \langle c^g\mid g\in \bar G\rangle$, because $V$ is a minimal normal subgroup of $G\langle\f\rangle$, while the right-hand side is both $G$- and $\f$-invariant. Hence, $[V,Z]=0$ or, in other words,  $\bar Z=1$, so that $\bar G$ has $m$-bounded order.
\end{proof}

We now finish the proof of Theorem~\ref{t-mainr}. Recall that any element of $G/F$ acts nontrivially on at least one of the factors $V_i$ in \eqref{e-prod}. Let $1\ne c_1\in C_{V_1}(\f)$, which exists by \eqref{e-clos}. Then the right Engel sink of $c_1$ is non-trivial:
 $\mathscr R(c_1)\ne \{1\}$. Indeed, otherwise $c_1$ belongs to the hypercentre $\zeta_{\infty}(G)$   by Baer's theorem \cite[12.3.7]{rob}, and then  $V_{1}\leq  \zeta_{\infty}(G)$  by the minimality of $V_{1}$. This, however, contradicts Lemma~\ref{e-ch0}. Therefore we can choose $a_1\in G$ such that
 $$
 [c_1,\,{}_na_1]\ne 1 \quad \text{for any }n\in \N.
 $$

 If $C_G(V_1)=F$, then $|G/F|$ is $m$-bounded by Lemma~\ref{l-mainr} and the proof is complete. Otherwise $C_G(V_1)$ acts non-trivially on at least one of the remaining factors in \eqref{e-prod}. Without loss of generality, we can assume that $[V_2,C_G(V_1)]\ne 1$. Let  $1\ne c_2\in C_{V_2}(\f)$, which exists by \eqref{e-clos}. Then the right Engel sink of $c_2$ in the product $V_2C_G(V_1)$ is non-trivial:
 $\mathscr R_{V_2C_G(V_1)}(c_2)\ne \{1\}$. Otherwise $c_2$ belongs to the hypercentre $\zeta_{\infty}(V_{2}C_G(V_1))$ of the semidirect product $V_{2}C_G(V_1)$ by Baer's theorem \cite[12.3.7]{rob}, and since $V_{2}C_G(V_1)$ is a normal subgroup of $G\langle\f\rangle$, then $V_2\leq \zeta_{\infty}(V_{2}C_G(V_1))$ by the minimality of $V_2$, which contradicts Lemma~\ref{e-ch0}. Therefore we can choose $a_2\in C_G(V_1)$ such that
 $$
 [c_2,\,{}_na_2]\ne 1 \quad \text{for any }n\in \N.
 $$
 Note that at the same time $[c_1,a_2]=1$.

 If $C_G(V_1V_2)=C_G(V_1)\cap C_G(V_2)=F$, then $|G/F|$ is $m$-bounded by Lemma~\ref{l-mainr} and the proof is complete. Otherwise $C_G(V_1V_2)$ acts non-trivially on at least one of the remaining factors in \eqref{e-prod}, say on $V_3$.  Let  $1\ne c_3\in C_{V_3}(\f)$, which exists by \eqref{e-clos}. The right Engel sink of $c_3$ in the product $V_3C_G(V_1V_2)$ is non-trivial:
 $\mathscr R_{V_3C_G(V_1V_2)}(c_3)\ne \{1\}$, as otherwise $c_3\in \zeta_{\infty}(V_{3}C_G(V_1V_2))$, whence  $V_3\leq \zeta_{\infty}(V_{3}C_G(V_1V_2))$ by the minimality of $V_3$, contrary to Lemma~\ref{e-ch0}. Therefore we can choose $a_3\in C_G(V_1V_2)$ such that
 $$
 [c_3,\,{}_na_3]\ne 1 \quad \text{for any }n\in \N.
 $$
 Note that at the same time $[c_1,a_3]=1$ and $[c_2,a_3]=1$.

 We can continue this construction in the obvious fashion as long as $C_G(V_1V_2\cdots V_k)\ne F$. We claim that this process will stop at  $C_G(V_1V_2\cdots V_k)= F$ for some $k\leq m$. Indeed, otherwise we would have  constructed elements $c_i\in V_i$ and $a_1\in G$ and $a_i\in C_G(V_1\cdots V_i)$ for $i=2,\dots ,m+1$ such that
 $$
 [c_i,\,{}_na_i]\ne 1 \quad \text{for any }n\in \N,\qquad\text{while}\quad  [c_i, a_j]=1\quad\text{for any }j>i.
 $$
Then the product $c=c_1c_2\cdots c_{m+1}$ would have right Engel sink $\mathscr R(c)$ of cardinality at least $m+1$, since  for any given $n_i\in\N$, $i=1,2,\dots,m+1$,   the $m+1$  elements  $[c,\,{}_{n_i}a_i]$ are different. Indeed,  as an element of the product $V=V_1\times\dots\times V_k$,  the commutator $[c,\,{}_{n_i}a_i]$  has trivial components in the factors $V_1,\dots V_{i-1}$ and a non-trivial component in $V_i$. The inequality $|\mathscr R(c)|\geq m+1$ contradicts the hypothesis of the theorem.

This contradiction shows that  $C_G(V_1V_2\cdots V_k)= F$ for some $k\leq m$.
 Then
 $$
 \begin{aligned}
  |G/F|=|G/C_G(V_1\cdots V_k)|&= \big|G/\big(C_G(V_1)\cap\cdots \cap C_G(V_k)\big)\big|\\
  &\leq |G/C_G(V_1)|\cdots |G/C_G(V_k)|
 \end{aligned}
 $$
and we obtain that   $|G/F|$ is $m$-bounded by Lemma~\ref{l-mainr}.
 \end{proof}

 \section*{Acknowledgements}
The first author was supported by the Mathematical Center in Akademgorodok, the agreement with Ministry of Science and High Education of the Russian Federation no.~075-15-2019-1613. The second author was supported by FAPDF and CNPq-Brazil.


\begin{thebibliography}{99}
 \bibitem{acc-khu-shu}    {C.~Acciarri, E.~I.~Khukhro and  P.~Shumyatsky},  Profinite groups with an automorphism whose fixed points are right Engel, \textit{Proc. Amer. Math. Soc.} \textbf{147}, no.~9 (2019), 3691--3703.



\bibitem{acc-shu-sil18} C. Acciarri, P. Shumyatsky, and D. S. Silveira, On groups with automorphisms whose fixed points are Engel,    \textit{Ann. Matem. Pura Appl.} \textbf{197} (2018), 307--316.


  \bibitem{acc-shu-sil19}  C. Acciarri, P. Shumyatsky, and D. Silveira,
Engel sinks of fixed points in finite groups, \textit{J.~Pure Appl. Algebra} \textbf{223}, no.~11 (2019), 4592--4601.


\bibitem{acc-sil18} C. Acciarri and D. Silveira,
Profinite groups and centralizers of coprime automorphisms whose elements are Engel, \textit{J.~Group Theory} \textbf{21}, no.~3 (2018), 485--509.


\bibitem{acc-sil20} C. Acciarri and D. Silveira, Engel-like conditions in fixed points of automorphisms of profinite groups, \textit{Ann. Matem. Pura Appl. (4)} \textbf{199}, no.~1 (2020), 187--197.


\bibitem{gol} E. S. Golod, On nil-algebras and residually finite $p$-groups, \emph{Izv. Akad.
Nauk SSSR Ser. Mat.} {\bf  28} (1964), 273--276; English transl. in \emph{Amer. Math. Soc. Translations}
(2) {\bf 48} (1965), 103--106.

\bibitem{gor}  D. Gorenstein,
\textit{Finite groups},  2nd ed., Chelsea, 1980.


\bibitem{gur-tra} R. M. Guralnick and G. Tracey,
On the generalized Fitting height and insoluble length of finite groups,
\emph{Bull. London Math. Soc.} {\bf 00} (2020) 1--8, \verb#doi:10.1112/blms.12372#.

\bibitem{h-m} B.~Hartley and T.~Meixner, Finite soluble groups
containing an element of prime order whose centralizer is small, \emph{Arch. Math. (Basel)} {\bf
 36} (1981), 211--213.

\bibitem{hei} H. Heineken, Eine Bemerkung \"uber engelsche Elemente, \emph{Arch. Math. (Basel)} {\bf 11} (1960), 321.

\bibitem{hup} B. Huppert, \textit{Endliche Gruppen}. I,
Springer, Berlin, 1967.

\bibitem{kms}  {E. I. Khukhro,  N.\,Yu.~Makarenko,  and P.~Shumyatsky},   Frobenius groups of automorphisms and their fixed points, \textit{Forum Math.} {\bf  26} (2014), 73--112.

\bibitem{khu-shu16}     {E.~I.~Khukhro and  P.~Shumyatsky},   Almost Engel finite and profinite groups, \textit{Int. J. Algebra Comput.}
{\bf  26}, no.~5 (2016), 973--983.


\bibitem{khu-shu18} E. I. Khukhro and P. Shumyatsky, Almost Engel compact groups, \emph{J.~Algebra} {\bf  500} (2018), 439--456.

         \bibitem{khu-shu18a}    {E.~I.~Khukhro and  P.~Shumyatsky},   Finite groups with Engel sinks of bounded rank, \textit{Glasgow Math.~J.} {\bf 60}, no.~3 (2018),   695--701.

     \bibitem{khu-shu19}    {E.~I.~Khukhro and  P.~Shumyatsky}, Compact groups   all elements of which are almost right Engel, \textit{Quart. J. Math} \textbf{70} (2019), 879--893.

     \bibitem{khu-shu20}    E. I. Khukhro and  P. Shumyatsky,  Compact groups with countable Engel sinks, \textit{Bull. Math. Sci.} \textbf{9}, no.~2 (2020) 2050015.

         \bibitem{khu-shu20a}     E.~I.~Khukhro and  P.~Shumyatsky,
      Compact groups in which all elements have countable right Engel sinks, {\it submitted},
      2020. \verb#https://arXiv:  https://arxiv.org/pdf/2004.11680.pdf#.

 \bibitem{khu-shu-tra}    {E.~I.~Khukhro,  P.~Shumyatsky, and G. Traustason}, Right Engel-type subgroups and length parameters of finite groups, \textit{J. Austral. Math. Soc.},   2019; \verb#DOI: 10.1017/S1446788719000181#.




    \bibitem{khu-shu202} E. I. Khukhro and P. Shumyatsky, On profinite groups with automorphisms  whose fixed points have countable  Engel sinks, \textit{submitted} (2020),  \verb#arXiv:2006.05959#.


\bibitem{rob} D. J. S. Robinson, \emph{A course in the theory of groups}, Springer, New York, 1996.

    \bibitem{rod-shu} S. Rodrigues and P. Shumyatsky, Exponent of a finite group admitting a
coprime automorphism, \textit{J.~Pure Appl. Algebra} \textbf{224}, no.~9 (2020),  article 106370.

\bibitem{tho}
{ J.~Thompson}, Finite groups with fixed-point-free
automorphosms of prime order, {\it Proc. Nat. Acad. Sci. U.S.A.}
{\bf 45} (1959), 578--581.

\bibitem{wan-che} Y. M. Wang and Z.~M.~Chen,  Solubility of finite groups admitting a coprime order operator group, \emph{Boll. Un. Mat. Ital. A (7)} {\bf 7}, no.~3 (1993), 325--331.

\bibitem{wi-ze} J. S. Wilson and E. I. Zelmanov, Identities for Lie algebras of pro-$p$ groups, \emph{J. Pure Appl. Algebra} {\bf 81}, no.~1 (1992), 103--109.
\end{thebibliography}
\end{document}